\documentclass[11pt,centertags]{amsart}
\usepackage{amssymb}            
\usepackage{mathrsfs}             
\usepackage{hyperref}           
\usepackage{ifsym}              
\usepackage{calc}               
\usepackage{enumerate}          
\usepackage[all]{xy}            
\usepackage{psfrag}             
\usepackage[usenames]{color}    
\usepackage{verbatim}           
\usepackage{fancyhdr}           
\usepackage{url}                
\usepackage{gloss}              
\usepackage{yhmath}             

\newif\ifPDF
\ifx\pdfoutput\undefined
    \PDFfalse
\else
    \ifnum\pdfoutput > 0
        \PDFtrue
    \else
        \PDFfalse
    \fi
\fi \ifPDF
    \usepackage[pdftex]{graphicx}
    \DeclareGraphicsExtensions{.pdf,.png,.jpg}
    \graphicspath{{}}
\else
    \usepackage{graphicx}
    \DeclareGraphicsExtensions{.eps}
    \graphicspath{{}}
\fi

\newtheorem*{main*}{Main Theorem}

\newtheorem{theorem}{Theorem}[section]
\newtheorem*{theorem*}{Theorem}

\newtheorem{lemma}[theorem]{Lemma}
\newtheorem{corollary}[theorem]{Corollary}

\newtheorem*{question*}{Question}

\newtheorem*{conjecture*}{Conjecture}

\theoremstyle{definition}

\newtheorem*{definition*}{Definition}

\theoremstyle{remark}
\newtheorem{remark}[theorem]{Remark}

\newtheorem{acknowledgement}[theorem]{Acknowledgement}

\numberwithin{equation}{section}

\newcommand{\E}{\mathbb{E}}

\newcommand{\R}{\mathbb{R}}

\newcommand{\Z}{\mathbb{Z}}

 \DeclareMathOperator{\Sp}{\Sp}

\newcommand{\union}{\cup}
\newcommand{\intersect}{\cap}

\newcommand{\pa}{\partial }

\providecommand{\to}{\longrightarrow }

\def\[#1\]{\begin{align*}\begin{split} #1 \end{split}\end{align*} }
\def\[[#1\]]{\begin{align}\begin{split} #1 \end{split}\end{align} }

\renewcommand{\bar}{\overline}

\newcommand{\paT}{\pa_{\mathrm{T}}}
\newcommand{\bT}{\mathrm{B}_{\mathrm{T}}}
\newcommand{\dT}{\mathrm{d}_{\mathrm{T}}}
\newcommand{\Fix}{\mathrm{Fix}}
\newcommand{\Min}{\mathrm{Min}}
\newcommand{\Isom}{\mathrm{Isom}}
\newcommand{\Homeo}{\mathrm{Homeo}}

\begin{document}
\author[Khek Lun Harold Chao]{Khek Lun Harold Chao}
\title[{CAT}(0) spaces with boundary the join of two Cantor sets]{{CAT}(0) spaces with boundary the join of two Cantor sets}
\email{khchao@indiana.edu}
\address{Department of Mathematics, Indiana University, Bloomington, IN 47405, USA}
\begin{abstract}
We will show that if a proper complete CAT(0) space $X$ has a visual
boundary homeomorphic to the join of two Cantor sets, and $X$ admits
a geometric group action by a group containing a subgroup isomorphic to $\Z^2$, then its Tits boundary is
the spherical join of two uncountable discrete sets. If $X$ is
geodesically complete, then $X$ is a product, and the group has a finite index subgroup isomorphic to a lattice in the product of two isometry groups of bounded valence bushy trees.
\end{abstract}
\maketitle

\section{Introduction}
CAT(0) spaces with homeomorphic visual boundaries can have very
different Tits boundaries. However, if $X$ admits a proper and
cocompact group action by isometries, or a geometric group action in short, then
this places a restriction on the possible Tits boundaries for a
given visual boundary. (We follow the definition of a proper group
action in Chapter I.8 of \cite{BH}; some use the term ``properly
discontinuous'' for this.) Kim Ruane has showed in \cite{Ru06} that
for a CAT(0) space $X$ with boundary $\pa X$ homeomorphic to the
suspension of a Cantor set, if it admits a geometric
group action, then the Tits boundary $\paT X$ is isometric to the
suspension of an uncountable discrete set. In this paper we will
show the following.
\begin{theorem}\label{mainthm1}
If a CAT(0) space $X$ has a boundary $\pa X$ homeomorphic to the
join of two Cantor sets, $C_1$ and $C_2$, and if $X$ admits a geometric group action by a group containing a subgroup isomorphic to $\Z^2$, then its Tits boundary $\paT X$ is
isometric to the spherical join of two uncountable discrete sets. So
if $X$ is geodesically complete, then $X = X_1 \times X_2$ with $\pa
X_i$ homeomorphic to $C_i$, $i=1,2$.
\end{theorem}

As for the group acting on $X$, we will prove the following.
\begin{theorem}\label{mainthm2}
Let $X$ be a geodesically complete CAT(0) space such that $\pa X$ is
homeomorphic to the join of two Cantor sets. Then for a group $G < \Isom(X)$
acting geometrically on $X$ and containing a subgroup isomorphic to $\Z^2$, either $G$ or a subgroup of
$G$ of index 2 is a uniform lattice in $\Isom(X_1) \times
\Isom(X_2)$. Furthermore, a finite index subgroup of $G$ is a
lattice in $\Isom(T_1)\times \Isom(T_2)$, where $T_i$ is a bounded valence bushy tree quasi-isometric to $X_i$, $i=1,2$.
\end{theorem}

\begin{remark}
The assumption that $G$ contains a subgroup isomorphic to $\Z^2$ is only used to obtain a hyperbolic element in $G$ with endpoints in  $\pa X \setminus (C_1 \union C_2)$, which we use in Section 4 to prove Theorem \ref{mainthm1}. It is conjectured that a CAT(0) group is either Gromov hyperbolic or it contains a subgroup isomorphic to $\Z^2$.  Without using the assumption on $G$, we can show that $G$ cannot be hyperbolic, which follows from Lemma \ref{no rank 1} and the Flat Plane Theorem. (\cite{BH}, Theorem III.H.1.5) Thus if the conjecture is shown to be true for general CAT(0) groups, the assumption on $G$ will not be necessary. The conjecture has been proved for some classes of CAT(0) groups, see \cite{KapKlein} and \cite{CaprHagl09} for examples.

If $X_i$ are proper geodesically complete, one might hope that they
are trees, so $G$ will be a uniform lattice in the product of two
isometry groups of trees. Surprisingly, this may not be the case.
Ontaneda constructed a 2-complex $Z$ which is non-positively curved
and geodesically complete with free group $F_n$ as its fundamental
group. (See proof of proposition 1 in \cite{Ont}) Its universal
cover is quasi-isometric to $F_n$, so it is a Gromov hyperbolic
space with Cantor set boundary, while being also a CAT(0) space. Under an additional condition that the isotropy subgroup of
$\Isom(X_i)$ of every boundary point of $X_i$ acts cocompactly on
$X_i$, then $X_i$ is a tree. (See proof of Theorem 1.3 in
\cite{CM1}.)

There are irreducible
lattice in a product of two trees, so $G$ may not have a finite index subgroup which splits as a product. See \cite{BM} for a detailed investigation.
\end{remark}

\begin{acknowledgement}
I would like to thank my advisor Chris Connell for suggesting this
problem to me and providing me with a lot of valuable discussions,
assistance and encouragements while I was on this project.
\end{acknowledgement}

\section{Preliminaries}
First we fix the notations. For a CAT(0) space $X$, its (visual)
boundary with the cone topology is $\pa X$. For a subset $H\subset X$, we denote by $\pa H := \bar H \cap \pa X$, where the closure $\bar H$ is taken in $\bar X := X \union \pa X$. The angular and the Tits
metrics on the boundary are denoted as $\angle(\cdot,\cdot)$ and
$\dT(\cdot,\cdot)$ respectively. We denote the boundary with the
Tits metric by $\paT X$. If $g$ is a group element acting on $X$ by
isometry, we denote by $\bar g$ the action of $g$ extended to $\pa
X$ by homeomorphism. If $g$ acts on $X$ by a hyperbolic isometry, the two endpoints of its axes on $\pa X$ are denoted by  $g^{\pm\infty}$. We refer to \cite{BH} for details on basic facts about CAT(0) spaces.

Let $X$ be a complete CAT(0) space with $\pa X$ homeomorphic to the
join of two Cantor sets $C_1$ and $C_2$, and $G < \Isom(X)$ be a group acting
on $X$ geometrically. We will not assume that $G$ contains a subgroup isomorphic to $\Z^2$ until Section 4. By the following lemma, we can
assume that $G$ stabilizes $C_1$ and $C_2$.

\begin{lemma}\label{subgroup stabilizing C_1, C_2}
Either $G$ or a subgroup of $G$ of index 2 stabilizes each of $C_1$ and
$C_2$.
\end{lemma}
\begin{proof}
Consider $\pa X$ as a complete bipartite graph with $C_1, C_2$ as
the two sets of vertices. For any $g\in G$, if $\bar g
\cdot x_1 \in C_1$ for some $x_1\in C_1$, then $\bar g \cdot C_i =
C_i$, $i=1,2$; otherwise $\bar g \cdot C_1 = C_2$ and $\bar
g \cdot C_2 = C_1$. So the homomorphism from $G$ to
symmetric group on two elements is well-defined and its kernel is
the subgroup of $G$ which stabilizes each of $C_1$ and $C_2$.
\end{proof}

By an arc we specifically mean a segment from a point in $C_1$ to a
point in $C_2$ which does not pass through any other point of $C_1$
or $C_2$, and by open (closed) segment a segment on the boundary
excluding (including) its two endpoints. We will investigate the
positions of the endpoints of hyperbolic elements in $G$.

We quote a basic result on dynamics on CAT(0) space boundary by Ruane:

\begin{lemma}[Ruane, \cite{Ru01} Lemma 4.1]\label{accum pt half flat}
Let $g$ be a hyperbolic isometry of a CAT(0) space $X$ and let $c$
be an axis of $g$. Let $z \in\pa X$, $z \neq g^{-\infty}$ and let
$z_i = \bar g^i \cdot z$. If $w \in\pa X $ is an accumulation point of
the sequence $(z_i)$ in the cone topology, then
$\angle(g^{-\infty},w)+\angle(w,g^{\infty})=\pi$, and
$\angle(g^{-\infty},z)=\angle(g^{-\infty},w)$. If $w\neq g^\infty$,
then $\dT(g^{-\infty}, w)+\dT(w,g^{\infty} )=\pi$.  In this case $c$
and a ray from $c(0)$ to $w$ span a flat half plane, and
$\dT(g^{-\infty},z)=\dT(g^{-\infty},w)$.
\end{lemma}

Recall that a hyperbolic isometry is of rank one if none of its axes
bounds a flat half plane, and it is of higher rank otherwise.

\begin{lemma}\label{no rank 1}
There is no rank one isometry in $G$.
\end{lemma}
\begin{proof}
Take any $g \in G$. Assume without loss of generality that
$g^\infty\in \pa X \setminus C_2$. Then for any point $y\in C_2$,
$\bar g^n\cdot y$ cannot accumulate at $g^\infty$ since $C_2$ is
closed in $\pa X$. Any accumulation point of $\bar g^n\cdot y$ will
form a boundary of a half plane with $g^{\pm\infty}$ by Lemma
\ref{accum pt half flat}. So $g$ is not rank one.
\end{proof}

We note also that no finite subset of points on the boundary is
stabilized by $G$, which readily follows from a result by Ruane, quoted in a paper
by Papasoglu and Swenson, and the fact that our $\pa X$ is not a suspension.

\begin{lemma}[Ruane, \cite{PaSw09} Lemma 26]\label{no finite orbit}
If $G$ virtually stabilizes a finite subset $A$ of $\pa X$, then $G$
virtually has $\Z$ as a direct factor. In this case $\pa X$ is a
suspension.
\end{lemma}

\section{Endpoints of a hyperbolic element}

We will show that there is no hyperbolic element of $G$ with one of
its endpoints in $C_1$ but not the other one. We will proceed by
contradiction, using as a key result the following theorem by
Papasoglu and Swenson to $\pa X$, itself a strengthening of a
previous result by Ballmann and Buyalo \cite{BaBu}. This theorem is
applicable to our $\pa X$ in light of the previous lemmas.

\begin{theorem}[Papasoglu and Swenson, \cite{PaSw09} Theorem
22]\label{PS}
If the Tits diameter of $\pa X$ is bigger than $\frac
{3\pi} 2$ then $G$ contains a rank 1 hyperbolic element. In
particular: If $G$ does not fix a point of $\pa X$ and does not have
rank 1, and $I$ is a (minimal) closed invariant set for the action
of $G$ on $\pa X$, then for any $x\in\pa X$, $\dT (x, I)\leq \frac
\pi 2$.
\end{theorem}

We put the word minimal in parentheses as it is not a necessary
condition, for if $I \subset \pa X$ is a closed invariant set, then
it contains a minimal closed invariant set $I'$, and so for any
$x\in \pa X$, $\dT(x,I) \leq \dT (x, I')\leq \frac \pi 2$.

Note that the above theorem implies that $\pa X$ has finite Tits
diameter, and hence the CAT(1) space $\paT X$ is connected.

Now assume that $g\in G$ is hyperbolic such that $g^\infty \in C_1$ and
$g^{-\infty} \in \pa X \setminus C_1$.

\begin{lemma}\label{contains F}
$\Fix(\bar g)$ contains boundary of a 2-flat.
\end{lemma}
\begin{proof}
By Lemma \ref{no rank 1}, $g^{\pm\infty}$ bound a half plane, so
there is a segment joining $g^{\pm\infty}$ fixed by $g$, then it is
contained in $\pa \Min(g)$. Then by Theorems 3.2 and 3.3 of
\cite{Ru01}, $\Min(g) = Y \times \R$ with $\pa Y \neq \varnothing$,
and $C(g)/\langle g \rangle$ acts on the CAT(0) space $Y$
geometrically. Since $Y$ has nonempty boundary, so by
 Theorem 11 of \cite{Sw99} there is a hyperbolic element in
$C(g)/\langle g \rangle$ which has an axis in $Y$ with two endpoints
on $\pa Y$. Thus there is a 2-flat in $\Min(g)$.
\end{proof}
Denote this 2-flat by $F$, and let $z$ be a point in $\pa F \cap
C_1$ other than $g^\infty$.

\begin{figure}[htb]
  \centering
  \def\svgwidth{200pt}
  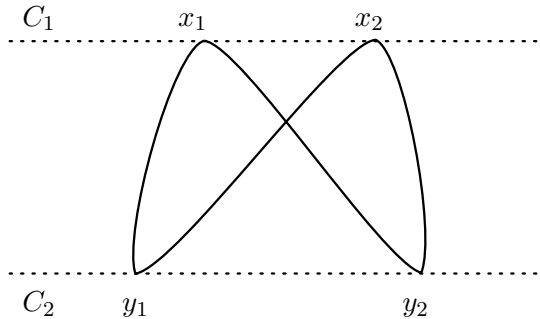
  \caption{Boundary of a 2-flat in $\Min(h)$}
\end{figure}

\begin{lemma}
If $F_0$ is a 2-flat whose boundary is contained in $\Fix(\bar h)=\pa
\Min(h)$ for some hyperbolic $h\in G$, then $\pa F_0$ intersects each of $C_1$ and
$C_2$ at exactly 2 points.
\end{lemma}
\begin{proof}
Suppose not, then denote the points at which $\pa F_0$ alternatively
intersects $C_1$, $C_2$ by $x_1, y_1, x_2, y_2, \ldots, x_n, y_n$.
Consider the segment joining $x_1$ and $y_2$. We may assume that not
both of $x_1$, $y_2$ are endpoints of $h$. (If not, choose $y_1$ and
$x_3$ instead.) From the assumption on $\pa F_0$, this segment is
not part of $\pa F_0$. Its two endpoints are fixed, but the arc
joining them is not in $\Fix(\bar h)$ because $\Fix(\bar h)$ is a
suspension with suspension points $h^{\pm\infty}$. However, this arc
is stabilized by $h$ because of the cone topology of $\pa X$. Action
of $G$ on $\paT X$ is by isometries. Take a point $p$ in the open
arc between $x_1$ and $y_2$. Since $\paT X$ is connected there
exists a Tits segment in this arc from $p$ to one of $x_1$ and
$y_2$, say $x_1$. Choose a new point on this segment as $p$ if
necessary, we can assume $\dT(p,x_1) < \dT(y_2, x_1)$. Now
$\dT(\bar h\cdot p, \bar h \cdot x_1) = \dT(\bar h\cdot p, x_1)$ and $\bar h\cdot p$ is
also on the arc. $\bar h\cdot p$ cannot be on the open segment between
$p$ and $x_1$. If $\bar h\cdot p$ were on the open segment between $p$
and $y_2$, the Tits geodesic from $\bar h\cdot p$ to $x_1$ would go
through $p$ or $y_2$, both would contradict $\dT(\bar h\cdot p,
x_1)=\dT(p, x_1)$. So $\bar h\cdot p = p$. Then $p \in \pa\Min(h)$ and
lies on a path in $\pa \Min(h)$ joining $h^{\pm\infty}$, forcing the arc
 to be in $\pa\Min(h)$, which contradicts the previous assertion.
\end{proof}

Denote the segment in $\pa X$ from $g^\infty$ to $z$ passing through
$g^{-\infty}$ by $\beta$. Let $y$ be the point where $\beta$
intersects $C_2$. The essense of the following arguments is to look
for a point in $\paT X$ that is over $\pi/2$ away from $C_1$ or
$C_2$, which are closed $G$-invariant subsets, so obtaining a
contradiction to Theorem \ref{PS}.

\begin{lemma}
$g^{-\infty}$ cannot be on the closed segment in $\beta$ from
$g^\infty$ to $y$.
\end{lemma}
\begin{proof}
Suppose not. The Tits length of this segment from $g^\infty$ to $y$ is at least
$\pi$. Let $0<\delta<\pi/2$ be such that $2\delta \leq \dT(y, C_1)$. Take a point $p$ on this segment so
that $\dT(p,g^\infty)=\pi/2+ \delta$. Then $\dT(p,y) \geq \pi/2 - \delta$.
Now for any point $x\in C_1$ other than $g^\infty$, if the Tits
geodesic segment from $p$ to $x$ passes through $y$, then
$$\dT(p, x)
\geq \dT(p,y)+\dT(y,C_1) \geq (\pi/2 - \delta) + 2\delta = \pi/2 + \delta;$$
while if it passes through $g^\infty$,
then obviously $\dT(p, x) > \dT(p, g^\infty) = \pi/2 +\delta$. So
$\dT(p, C_1) \geq \pi/2 + \delta$, which contradicts Theorem \ref{PS}.
\end{proof}

Now we deal with the case that $g^{-\infty}$ is in the open segment
in $\beta$ from $y$ to $z$. We state a lemma first which will also
be used in later arguments.

\begin{lemma}\label{conv of forward seq}
Suppose $h \in G$ is a hyperbolic element such that $F_0 \subset
\Min(h)$ whose boundary intersects $C_1$ and $C_2$ alternatively at
$x_1, y_1, x_2, y_2$. Assume that the endpoint $h^{-\infty}$ is on some
 open arc, say the open arc between $x_i$ and $y_j$, while another
endpoint $h^\infty$ is not contained in the closed arc between $x_i$
and $y_j$. Then for any point $x\in C_1$ other than $x_1$ and $x_2$,
the sequence $\bar h^n \cdot x$ can only accumulate at $x_1$ or
$x_2$. Similarly, for any point $y\in C_2$ other than $y_1$ and
$y_2$, the sequence $\bar h^n \cdot x$ can only accumulate at $y_1$ or
$y_2$.
\end{lemma}
\begin{proof}
Suppose not, then the sequence has an accumulation point $x' \in
C_1\setminus\{x_1,x_2\}$. By Lemma \ref{accum pt half flat}, $x'$
forms boundary of a half flat plane with $h^{\pm\infty}$. This
boundary goes from $h^\infty$ to $x'$, and then passes through $x_i$
or $y_j$ before ending at $h^{-\infty}$. If it passes through $x_i$,
then the Tits length of segment on this boundary joining $h^\infty$
to $x_i$ is the total length of the half-plane boundary $\pi$ minus
the length of the segment from $x_i$ to $h^{-\infty}$, thus it is
equal to the length of the Tits geodesic segment on $\pa F_0$ joining
these two points, so there are two geodesics for these two points.
But this contradicts the uniqueness of Tits geodesic between two
points less than $\pi$ apart. If the boundary of the half flat plane
goes through $y_j$, apply the same argument to the points $h^\infty$
and $y_j$ and we have the same contradiction. For the case $y\in
C_2\setminus\{y_1,y_2\}$ use the same argument.
\end{proof}
\begin{lemma}\label{not on open arc}
$g^{-\infty}$ cannot be in the open segment from $y$ to $z$.
\end{lemma}
\begin{proof}
Suppose not. For any point $z'\in C_1$ other than $g^\infty$ and
$z$, the sequence $\bar g^{-n}\cdot z'$ converges to $z$ by Lemma
\ref{conv of forward seq} and Lemma \ref{accum pt half flat} which
says that $\bar g^{-n}\cdot z'$ cannot accumulate at $g^\infty$.

The segment $\beta$ has Tits length larger than $\pi$, so there is a
point $w\in \beta$ which is more than $\pi/2$ away from $g^\infty$
and from $z$.

By lower semi-continuity of the Tits metric,
\begin{align*}
\dT(w,z') & =\lim_{n\to\infty}\dT(\bar g^{-n}\cdot w, \bar
g^{-n}\cdot z') \\
& \geq \dT(\lim_{n\to\infty} \bar g^{-n}\cdot w,\lim_{n\to\infty}
\bar g^{-n}\cdot z') = \dT(w,z).
\end{align*}

So $\dT(w,C_1) > \pi/2$, a contradiction to Theorem
\ref{PS}.
\end{proof}
\begin{figure}[htb]
  \centering
  \def\svgwidth{350pt}
  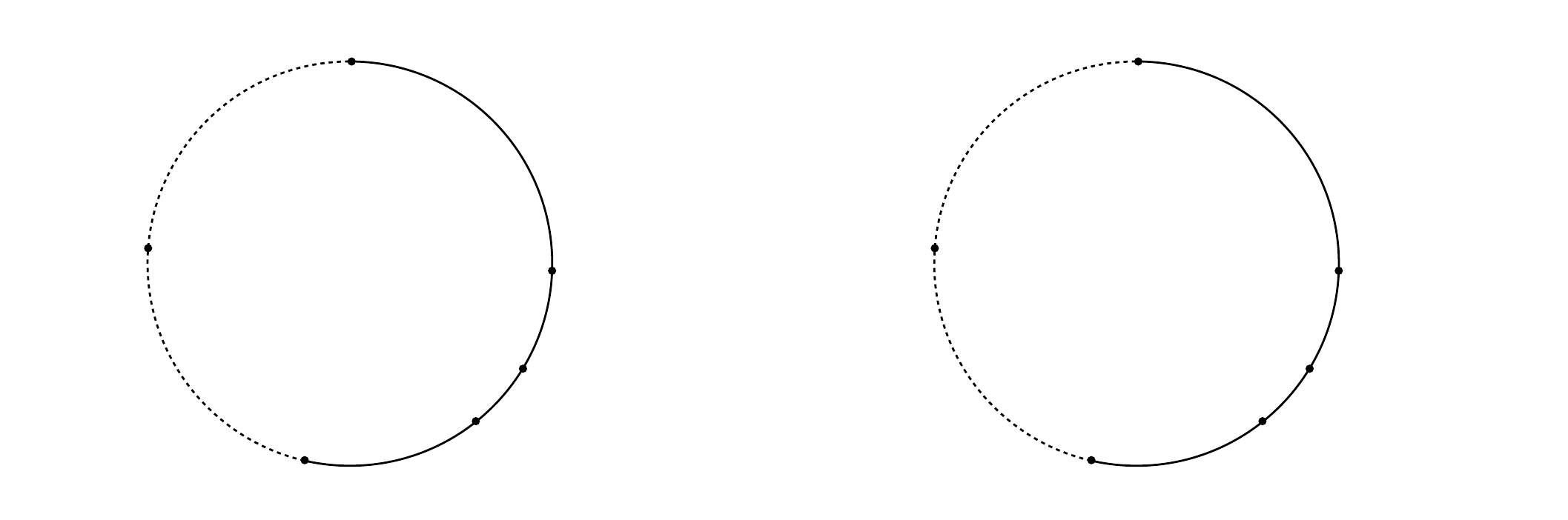
  \caption{$\pa F$ in Lemma \ref{not on open arc}}
\end{figure}

We see from these lemmas that the endpoints of a hyperbolic
element must be both in $C_1$, or both in $C_2$, or none is in $C_1
\union C_2$.

If $g$ is a hyperbolic element of $G$ with endpoints not in $C_1
\union C_2$, we have the following results.

\begin{lemma}\label{boundary is F}
$\pa\Min(g)$ is the boundary of a 2-flat.
\end{lemma}
\begin{proof}
Since $\pa \Min(g)$ is a suspension, so it can only be a circle or a
set of two points. However, as $g$ acts on $\paT X$ by isometry, we
see that $g$ must fix the arc on which $g^\infty$ lies. So $\pa
\Min(g) = \Fix(\bar g)$ can only be a circle. Then by the same
reason as in Lemma \ref{contains F} $\Min(g)$ contains a 2-flat,
whose boundary is the circle.
\end{proof}

Suppose for convenience that $g^\infty$ is on the open arc from
$x_1\in C_1$ to $y_1\in C_2$, and $x_2\in C_1$, $y_2\in C_2$ are the
two other points on the boundary $\pa F$.

\begin{lemma}
For $g$ as above, $g^{-\infty}$ can only be on the open arc from
$x_2$ to $y_2$.
\end{lemma}
\begin{proof}
Suppose $g^{-\infty}$ were not on this arc. Without loss of
generality let $g^{-\infty}$  be on the arc joining $y_1$ and $x_2$.
Now the segment from $x_1$ to $x_2$ through $y_1$ has Tits length
larger than $\pi$, so we can choose a point $p$ on this segment so
that $p$ is at distance more than $\pi/2$ away from $x_1$ and $x_2$.
By Lemma \ref{conv of forward seq}, for any other point $x'\in C_1$,
$\bar g^n\cdot x'$ cannot have an accumulation point other than
$x_1$ and $x_2$. Passing to a subsequence $\bar g^{n_k}\cdot x' \to
x_i$, $i=1$ or $2$, we have
\begin{align*}
\dT(p, x') & = \lim_{n_k\to\infty} \dT(\bar g^{n_k}\cdot p,\bar
g^{n_k}\cdot x') \\
& \geq \dT(\lim_{n_k\to\infty} \bar g^{n_k}\cdot
p,\lim_{n_k\to\infty} \bar g^{n_k}\cdot x')= \dT(p, x_i),
\end{align*} then $\dT(p, C_1)
> \pi /2$, contradicting Theorem \ref{PS}.
\end{proof}

\section{Main result}
Now we add the assumption that $G$ contains a subgroup isomorphic to $\Z^2$, then the Flat Torus Theorem (\cite{BH}, Theorem II.7.1) implies that there exists two commuting hyperbolic elements $g_1, g_2\in G$, such that $\Min(g_1)$, formed by the
axes of $g_1$, contains axes of $g_2$ not parallel to those of $g_1$. Then an axis of $g_1$ and an axis of $g_2$ span a 2-flat in $\Min(g_1)$, and  elements $g_1^n g_2^m$ are also
hyperbolic and have axes in this 2-flat with endpoints dense on the boundary of this 2-flat. So we can choose some
hyperbolic element $g$ so that its endpoints are not in $C_1 \union
C_2$.

We start with a lemma about the orbits of the group action, then we will prove Theorem \ref{mainthm1}.
\begin{lemma}\label{inf orbit}
For any two distinct points $w_1,w_2 \in \pa X$, there exists a sequence $(g_i)_{i=0}^\infty \subset G$ such that the points $\bar g_i \cdot w_j$, where  $0\leq i < \infty$ and $j\in\{1,2\}$, are distinct.
\end{lemma}
\begin{proof}
From Lemma \ref{no finite orbit} we know that every $w\in \pa X$ has an infinite orbit $G \cdot w$. So let $(h_i)_{i=0}^\infty \subset G$ be a sequence such that $\bar h_i \cdot w_1$ are distinct. We will construct the sequence $(g_i)$ inductively. First set $g_0 = e$.

Suppose that for $n \geq 0$ we have $g_0, \ldots, g_n$ such that $\bar g_i \cdot w_j$, where $0\leq i \leq n$, $j\in\{1,2\}$, are distinct.  Let $S_n :=  \{\bar g_m \cdot w_1, \bar g_m \cdot w_2: 0\leq m \leq n \}$ . Pass to a subsequence of  $(h_i)$ so that $\bar h_i \cdot w_1 \notin S_n$. (We will keep denoting any subsequence by $(h_i)$.) If there exists some $h_j$  such that $\bar h_j \cdot w_2 \notin S_n$, then set $g_{n+1} = h_j$. Otherwise, there exists some $\bar g_m \cdot w_k \in S_n$ such that $\bar h_i \cdot w_2 = \bar g_m \cdot w_k$ for infinitely many $h_i$. Pass to this subsequence. Since the orbit of $\bar g_m \cdot w_k$ is infinite, there exists $h' \in G$ such that $\bar {h'} \cdot (\bar g_m \cdot w_k) \notin S_n$, so $\bar {h' h_i} \cdot w_2 \notin S_n$. Now  $\bar {h' h_i} \cdot w_1 \notin S_n$ for infinitely many $h_i$. Set $g_{n+1} = h' h_i$ for one such $h_i$. Hence we get the desired sequence $(g_i)$.
\end{proof}
\begin{remark}
The only condition required on the group action is that every orbit is infinite. This proof can be used to show a similar result for any finite set $\{w_1, \ldots w_n\}$.
\end{remark}
\begin{lemma}\label{pi/2 apart}
For any $x\in C_1$, $y\in C_2$ we have $\dT(x,y) = \pi / 2$. Hence
$\paT X$ is metrically a spherical join of $C_1$ and $C_2$.
\end{lemma}
\begin{proof}
Consider some $g\in G$ which is hyperbolic with endpoints not on
$C_1\union C_2$. Let $\pa \Min(g)=\pa F$. We will first prove that for
 $x_1,x_2\in C_1 \cap \pa F$, $y_1,y_2\in C_2 \cap \pa F$, we have $\dT(x_i,y_j) = \pi / 2$, where $i,j=1,2$. Take
any of the four arcs making up $\pa F$, say the arc joining $x_1$ and $y_1$.

The endpoints of hyperbolic elements in $Z_g$ are dense on $\pa F$,
so we can pick a $g' \in Z_g$ so that $g'^{-\infty}$ is as close to
the midpoint of arc $x_2$ and $y_2$ as we want. Let $0<\delta<
\min(\dT(x_2, C_2),\dT(y_2, C_1))$. Pick $g'$ so that $\left|
\dT(g'^{-\infty}, x_2) - \dT(g'^{-\infty}, y_2) \right| < \delta$. For
any point $x\in C_1$ other than $x_2$, if the Tits geodesic segment
from $g'^{-\infty}$ to $x$ passes through $y_2$, then
\begin{align*}\dT(g'^{-\infty},x)
& \geq \dT(g'^{-\infty},y_2)+\dT(y_2,C_1) \\
& > \dT(g'^{-\infty},x_2) - \delta + \dT(y_2,C_1) >
\dT(g'^{-\infty},x_2);
\end{align*}
 while if it passes through $x_2$ then obviously $\dT(g'^{-\infty},x) >
 \dT(g'^{-\infty},x_2)$. For any $y\in C_2$ other than $y_2$, by similar reasoning on the Tits geodesic segment from $g'^{-\infty}$ to $y$, we
 have $\dT(g'^{-\infty},y) > \dT(g'^{-\infty},y_2)$.

For any arc joining $x\neq x_2 \in C_1$ and $y\neq y_2 \in
C_2$, since $\dT(g'^{-\infty},x) > \dT(g'^{-\infty},x_2)$, the point
$x_2$ cannot be an accumulation point of $\bar {g'}^n \cdot x$ by Lemma
\ref{accum pt half flat}, then by Lemma \ref{conv of forward seq}, $\bar {g'}^n \cdot x \to x_1$. Likewise, $\bar {g'}^n \cdot y
\to y_1$. So
\begin{align}
\dT(x, y) &= \lim_{n\to\infty} \dT(\bar {g'}^n \cdot x, \bar {g'}^n \cdot y) \label{>=d x1,y1} \\
&\geq \dT(\lim_{n\to\infty}\bar {g'}^n \cdot x, \lim_{n\to\infty}\bar {g'}^n \cdot y) = \dT(x_1, y_1). \notag
\end{align}
For any other arc joining $x_i$ to $y_j$ in $\pa F$, by lemma \ref{inf orbit} there exists $h \in G$ such that $\bar h \cdot x_i \neq x_2$ and $\bar h \cdot y_j \neq y_2$, so from the inequality (\ref{>=d x1,y1}) we get
$$\dT(x_i, y_j) = \dT(\bar h \cdot x_i, \bar h \cdot y_j) \geq \dT(x_1,y_1).$$

Thus all arcs have equal length $\pi/2$. Now for any $x \in C_1$,
$y\in C_2$, by Lemma \ref{conv of forward seq} the sequence
$\bar g^n\cdot x$ can accumulate at $x_1$ or $x_2$, and $\bar
g^n\cdot y$ can accumulate at $y_1$ or $y_2$, so passing to some
subsequence $(\bar g^{n_k})$, we have convergence sequences $\bar
g^{n_k}\cdot x \to x_i$ and $\bar g^{n_k}\cdot y \to y_j$. Then we
have inequality
\begin{align}\label{>=pi/2}
\dT(x, y) = \lim_{n_k\to\infty} \dT(\bar g^{n_k}\cdot x, \bar g^{n_k}\cdot
y) \geq \dT(x_i,y_j) = \pi/2.
\end{align}

Take a point $p$ on the open arc joining $x$ and $y$. Without loss
of generality assume that $p$ and $x$ are connected in $\paT X$ by a
segment in the arc. For any $\epsilon > 0$, we may choose a new point
on the segment from $p$ to $x$ to replace $p$ so that $0<\dT(x,p) <\epsilon$. Consider the Tits geodesic from
$p$ to some point in $C_2$. If it passes through $x$, then it
consists of the segment from $p$ to $x$ and an arc from $x$ to some
point in $C_2$, so by the inequality (\ref{>=pi/2}) its Tits length is at least $\pi/2+\dT(x,p)$. By
Theorem \ref{PS} $\dT(p,C_2) \leq \pi/2$, so there must be a Tits
geodesic from $p$ to some point in $C_2$ that does not pass through
$x$, hence it passes through $y$. Its length is at least $\dT(p,y)$, so $y$ is the closest point in $C_2$ to $p$, so $\dT(p,y) = \dT(p,C_2) \leq
\pi/2$. Then $\dT(x, y) \leq \dT(x,p) + \dT(p,y) < \pi/2 +
\epsilon$. Letting $\epsilon \to 0$ we have $\dT(x, y)\leq \pi/2$.
Combining with the inequality (\ref{>=pi/2}), $\dT(x,
y) =\pi/2$.
\end{proof}

\begin{theorem}\label{spherical join}
If $X$ is a CAT(0) space which admits a geometric group action by a group containing a subgroup isomorphic to $\Z^2$, and $\pa X$ is homeomorphic to the
join of two Cantor sets, then $\paT X$ is the spherical join of two
uncountable discrete sets. If $X$ is geodesically complete, i.e.
every geodesic segment in $X$ can be extended to a geodesic line,
then $X$ is a product of two CAT(0) space $X_1,X_2$ with $\pa X_i$
homeomorphic to a Cantor set.
\end{theorem}
\begin{proof}
We have shown that for any $x\in C_1$, $y\in C_2$, $\dT(x,y) = \pi
/2$ in Lemma \ref{pi/2 apart}, so every two distinct points in $C_i$
has Tits distance $\pi$ for $i=1,2$, i.e. $C_i$ with the Tits metric
is an uncountable discrete set. Then $\paT X$ is isomorphic to the
spherical join of $C_1$ and $C_2$, giving the first result. So with
the additional assumption that $X$ is geodesically complete, it
follows by Theorem II.9.24 of \cite{BH} that $X$ splits as a product
$X_1 \times X_2$, with $\pa X_i = C_i$ for $i=1,2$.
\end{proof}

\section{Some properties of the group}
We will show Theorem \ref{mainthm2} in this section. Assuming that $X$ is geodesically complete, and hence reducible by
Theorem \ref{spherical join}, we have the following result for the
group $G$. We do not require that $G$ stabilizes each of $C_1$ and
$C_2$ in this section.
\begin{theorem}\label{G to product of isom}
Let $X$ be a CAT(0) space such that $\pa X$ is homeomorphic to the
join of two Cantor sets and suppose $X$ is geodesically complete.
For a group $G < \Isom(X)$ containing $\Z^2$ and acting geometrically on $X$, either $G$ or a
subgroup of it of index 2 is a uniform lattice in $\Isom(X_1) \times
\Isom(X_2)$, where $X_1, X_2$ are given by Theorem \ref{spherical
join}.
\end{theorem}
\begin{proof}
We know from Theorem \ref{spherical join} that $X=X_1 \times X_2$,
so we only need to show that $G$ or a subgroup of it of index 2
preserves this decomposition.

By Lemma \ref{subgroup stabilizing C_1, C_2}, either
$G$ or a subgroup of it of index 2 stabilizes $C_1$ and $C_2$.
Replacing $G$ by its subgroup if necessary, we assume $G$ stabilizes
$C_1$ and $C_2$.

Denote by $\pi_i$ the projection of $X$ to $X_i$, $i=1,2$. Take any
$p_1, p_2\in X$ such that $\pi_2(p_1) = \pi_2(p_2)$. Extend
$[p_1,p_2]$ to a geodesic line $\gamma$, its projection to each of
$X_i$ is the image of a geodesic line. Since $X_1$ is totally
geodesic, the geodesic segment $[p_1,p_2]$ projects to a single
point $\pi_2(p_1)$ on $X_2$, i.e. a degenerated geodesic segment, so
$\pi_2(\gamma)$ is also a degenerated geodesic line. Thus the
endpoints $\gamma(\pm\infty)$ are in $C_1$. Now $\bar g \cdot
\gamma$ is a geodesics line passing through $\bar g \cdot p_1$,
$\bar g \cdot p_2$, and its endpoints $\bar g \cdot
\gamma(\pm\infty) \in C_1$, so $\pi_2(\bar g \cdot p_1) = \pi_2(\bar
g \cdot p_2)$. Similarly, for any $q_1, q_2\in X$ such that
$\pi_1(q_1) = \pi_1(q_2)$ we have $\pi_1(\bar g \cdot q_1) =
\pi_1(\bar g \cdot q_2)$. So $G$ preserves the decomposition $X =
X_1 \times X_2$, hence the result.
\end{proof}

We will show that $\Isom(X_i)$ is
isomorphic to a subgroup of $\Homeo(C_i)$ by the following lemma.

\begin{lemma}\label{map isom to homeo has finite kernel}
Suppose $X'$ is a proper complete CAT(0) space, and $G'<\Isom(X')$
acts properly on $X'$ by isometries.
\begin{enumerate}
  \item If $S\subset \pa X'$ is a set of points on the boundary such
  that the intersection $\bigcap_{w\in S} \overline{\bT(w,\pi/2)}$ is empty,
  then there exists a point $q\in X$ such that any non-hyperbolic
  $g\in \Isom(X')$ that fixes $S$ pointwise will fix $q$. In
  particular, such $g$ is elliptic.
    \item If $\pa X'$ is not a suspension and the radius of $\paT X'$
    is larger than $\pi/2$, then the map $G'\to \Homeo(\pa X')$, defined
    by extending the action of $G'$ to the boundary $\pa X'$, has a
    finite kernel, i.e. the subgroup of $G'$ that acts trivially on
    the boundary is finite. Moreover, assume the action of $G'$ is cocompact, then the kernel fixes a subspace of $X'$ with boundary $\pa X'$.
\end{enumerate}
\end{lemma}
\begin{proof}
To prove (1), observe that any such $g$ stabilizes all horospheres
and thus all horoballs centered at every $w\in S$. Take an arbitrary
point  $q'\in X$ and choose for each $w$ a closed horoball $H_w$
centered at $w$ that contains $q'$. Their intersection
$\bigcap_{w\in S} H_w$ is non-empty since it contains $q'$. By Lemma
3.5 of \cite{CM1} $\pa H_w =  \bar{\mathrm{B_T}(w,\pi/2)}$, then
$\pa (\bigcap_{w\in S} H_w) \subset \bigcap_{w\in S}(\pa H_w)
=\varnothing$. So $\bigcap_{w\in S} H_w$ is bounded. Also as every
$H_w$ is stabilized by $g$, so is $\bigcap_{w\in S} H_w$. As
$\bigcap_{w\in S} H_w$ is convex and compact, it contains a unique center
$q$, where the function $\sup \{\mathrm{d}_X(\cdot,z):z\in\bigcap_{w\in S} H_w \}$ is minimized. Then $g$
fixes $q$.

For (2), if $g\in G'$ acts by hyperbolic isometry, then $\pa \Min(g)
= \Fix(\bar g)$ is a suspension. Then any $g$ acting trivially on
the whole boundary $\pa X'$ is not hyperbolic. As $\paT X'$ has
radius larger than $\pi/2$, for every $x \in\pa X'$ there is some
$w\in\pa X'$ such that $\dT(x,w) > \pi/2$, so $x \notin
\bar{\mathrm{B_T}(w,\pi/2)}$, hence $S=\pa X'$ satisfies the
condition in (1). Now (1) implies that the kernel of $G' \to
\Homeo(\pa X')$ is a subgroup of the stabilizer of some point $q \in
X'$. As the action of $G'$ is proper, the kernel is finite.

Let $K$ be the kernel. The set fixed by $K$ is closed and convex. For any point $q$ fixed by the kernel, as $g \cdot q$ is fixed by $g K g^{-1} = K$, then $G' \cdot q$ is fixed by $K$. If the action of $G'$ is cocompact, then the set fixed by $K$ is quasi-dense, hence it is a subspace with boundary $\pa X'$.
\end{proof}

\begin{corollary}
Let $X$ be a geodesically complete CAT(0) space such that $\pa X$ is homeomorphic to the
join of two Cantor sets.  Then for a group $G < \Isom(X)$ containing $\Z^2$ and acting geometrically on $X$, either $G$ or a subgroup of it of
index 2 is isomorphic to a subgroup of $\Homeo(C_1) \times
\Homeo(C_2)$.
\end{corollary}
\begin{proof}
This follows from Theorem \ref{G to product of isom} and Lemma \ref{map isom to homeo has finite kernel}.
\end{proof}

We can still show this without the geodesic completeness assumption.

\begin{theorem}
Let $X$ be a CAT(0) space such that $\pa X$ is homeomorphic to the
join of two Cantor sets. Then for a group $G < \Isom(X)$ containing $\Z^2$ and acting geometrically on $X$, a finite quotient of either $G$ or a subgroup of
$G$ of index 2 is isomorphic to a subgroup in $\Homeo(C_1) \times
\Homeo(C_2)$.
\end{theorem}
\begin{proof}
Assume $G$ stabilizes each of $C_1$ and $C_2$ as in the proof of
Theorem \ref{G to product of isom}. Each $g\in G$ acts on $\pa X$ as
a homeomorphism, so it acts on $C_i \subset \pa X$ also as a
homeomorphism.

Suppose $g$ acts trivially on $C_1$ and $C_2$, i.e. $g$ is in the
kernel of $G\to \Homeo(C_1) \times \Homeo(C_2)$ . Then for any point
$x\in \pa X$ outside $C_1 \union C_2$, the arc on which $x$ lies is
a Tits geodesic segment of length $\pi/2$ in $\paT X$. Since $g$
acts on $\paT X$ by isometry and both endpoints of this Tits
geodesic segment are fixed by $g$, so $g$ fixes the whole arc, thus
$\bar g \cdot x = x$. Hence $g$ acts trivially on $\pa X$. One can
check that $\paT X$ has radius larger than $\pi/2$, so by Lemma
\ref{map isom to homeo has finite kernel} $G\to\Homeo(\pa X)$ has
finite kernel. Hence the result.
\end{proof}

In the case when $X$ is geodesically complete, actually we can prove
a stronger result, expressed in the last statement of Theorem
\ref{mainthm2}. Observe that $X_i$ is a Gromov hyperbolic space by
the Flat Plane Theorem, which states that a proper cocompact CAT(0)
space $Y$ is hyperbolic if and only if it does not contain a
subspace isometric to $\E^2$. Recall that a cocompact space is
defined as a space $Y$ which has a compact subset whose images under
the action by $\Isom(Y)$ cover $Y$. The (projected) action of $G$ on
$X_i$ is cocompact, even though the image in $\Isom(X_i)$ may not be discrete. As $\pa X_i$ does not contain
$S^1$, the result follows.

We will show $X_i$ is quasi-isometric to a tree. This is equivalent
to having  the \textit{Bottleneck Property} by a theorem of Manning,
which he proved with an explicit construction:
\begin{theorem}[\cite{Mann}, Theorem 4.6]\label{Manning}
Let $Y$ be a geodesic metric space. The following are equivalent:
\begin{enumerate}
  \item $Y$ is quasi-isometric to some simplicial tree $\Gamma$.
  \item (Bottleneck Property) There is some $\Delta > 0$ so that for all $x, y \in
  Y$ there is a midpoint $m = m(x, y)$ with $d(x, m) = d(y, m) = \frac 1 2 d(x,
  y)$ and the property that any path from $x$ to $y$ must pass
  within less than $\Delta$ of the point $m$.
\end{enumerate}
\end{theorem}

Pick a base point $p$ in $X_i$. There exists some $r>0$ such that $G \cdot
B(p,r)$ covers $X_i$.

\begin{lemma}\label{large ball separates space}
There exists $R >0$ such that for any $x,y$ in the same connected
component of $X_i \setminus B(p,R)$, the geodesic segment $[x,y]$
does not intersect $B(p,r)$.
\end{lemma}
\begin{proof}
Suppose on the contrary that for $R_n$ increasing to infinity, we
can find $x_n, y_n$ in the same connected component of $X_i
\setminus B(p,R_n)$ and $[x_n,y_n]$ intersects $B(p,r)$.  Since $\pa
X_i$ is compact in the cone topology, passing to a subsequence we
have $x_n \to \bar x$, $y_n \to \bar y$ for some $\bar x, \bar y\in
\pa X_i$. By \cite{BH} Lemma II.9.22, there is a geodesic line from
$\bar x$ to $\bar y$ intersecting $B(p,r)$. In particular, $\bar x
\neq \bar y$.

Since different connected components in the boundary of a hyperbolic
space correspond to different ends of the space (\cite{BH} Exercise
III.H.3.8), and $\pa X_i$ is a Cantor set, so $\bar x$ and $\bar y$
are in different ends of $X_i$, which are separated by $B(p,R_n)$
for $R_n$ large enough. But then $x_n$, $y_n$ will be in different
connected components of $X_i \setminus B(p,R_n)$, contradicting the
assumption. Hence the result.
\end{proof}

\begin{lemma}
$X_i$ has the Bottleneck Property.
\end{lemma}
\begin{proof}
For any $x,y\in X_i$, we may translate by some $g\in G$ so that the
midpoint $m$ of $[x,y]$ is in $B(p,r)$. We may assume that  $d(x,y)
> 2(R+r)$, then $x,y \in X_i \setminus B(p,R)$. By Lemma \ref{large
ball separates space}, $x,y$ are in different connected components
of $X_i \setminus B(p,R)$, hence any path connecting $x$ to $y$ must
intersect $B(p,R)$, so some point on this path is at a distance at
most $R+r$ from $m$. Thus the Bottleneck Property is satisfied.
\end{proof}

\begin{lemma}
$X_i$ is quasi-isometric to a bounded valence tree with no
terminal vertex.
\end{lemma}
\begin{proof}
First we describe briefly Manning's construction in his proof of
Theorem \ref{Manning}. Let $R'=20\Delta$. Start with a single point
$\star$ in $Y$. Call the vertex set containing this point $V_0$, and
let  $\Gamma_0$ be a tree with only one vertex and no edge, and
$\beta_0:\Gamma_0 \to Y$ be the map sending the vertex to $\star$.
Then for each $k \geq 1$, Let $N_{k-1}$ be the open $R$-neighborhood
of $V_{k-1}$. Let $\mathcal C_k$ be the set consists of path
components of $Y\setminus N_{k-1}$. For each $C\in \mathcal C_k$
pick some point $v$ at $C \intersect \bar N_k$. There is a unique
path component in $\mathcal C_{k-1}$ containing $C$, corresponding
to a terminal vertex $w \in V_{k-1}$. Connect $v$ to $w$ by a
geodesic segment. Let $V_k$ be the union of $V_{k-1}$ and the set of
new points from each of the path components in $\mathcal C_k$. Add
new vertices and edges to the tree $\Gamma_{k-1}$ accordingly to get
the tree $\Gamma_k$. Extend $\beta_{k-1}$ to $\beta_k$ by mapping
new vertices of $\Gamma_k$ to corresponding new vertices in $V_k$,
and new edges to corresponding geodesic segments. The tree $\Gamma=
\union_{k\geq 0}\Gamma_k$, and $\beta: \Gamma \to Y$ is defined to
be $\beta_k$ on $\Gamma_k$.

Apply the construction above to $X_i$. Since $X_i$ is geodesically
complete, each terminal vertex in $V_{k-1}$ will be connected by at
least one vertex in $V_k \setminus V_{k-1}$, and similarly so for
terminal vertices of $\Gamma_{k-1}$. So the tree $\Gamma$ has no
terminal vertex.

Manning proved that the length of each geodesic segment added in the
construction is bounded above by $R'+6\Delta$. Consider $w \in
V_{k-1}$ with corresponding path component $C_w \in \mathcal
C_{k-1}$. Every path component $C \in \mathcal C_k$ such that $C
\subset C_w$ gives a new segment joining $w$. Together with geodesic
completeness of $X_i$, this implies that such $C$ will contain at
least one path component of $X_i \setminus B(w,R'+6\Delta)$, and
every path component of $X_i \setminus B(w,R'+6\Delta)$ is contained
in at most one such $C$. (Geodesic completeness is used to ensure that no such $C$ will disappear when passing to  $X_i \setminus B(w,R'+6\Delta)$ .) Thus the number of new vertices in $V_k$
joining $w$ is bounded by the number of path components of $X_i
\setminus B(w,R'+6\Delta)$. Call the vertex in $\Gamma$
corresponding to $w$ as $p_w$. Since no more new segments will join
$w$ in subsequent steps, the degree of $p_w$ in $\Gamma$ equals one
plus the number of new vertices in $V_k$ joining $w$. Translate
$X_i$ by some $g$ so that $g \cdot w\in B(p,r)$. The number of path
components in $X_i \setminus B(w, R'+6\Delta)$ equals that in $X_i
\setminus B(g\cdot w, R'+6\Delta)$, which is at most the number of
path components in $X_i \setminus B(p, r+ R'+6\Delta)$, as $ B(g
\cdot w, R'+6\Delta) \subset B(p, r+ R'+6\Delta)$. Hence we obtain a
universal bound of the degree of $p_w$ in $\Gamma$, which means
$\Gamma$ has bounded valence.
\end{proof}

A tree of bounded valence with no terminal vertex is quasi-isometric
to the trivalent tree. Such tree is called a bounded valence \textit{bushy} tree.
Therefore we have shown the following:

\begin{theorem}
If $X_i$ is a proper cocompact and geodesically complete CAT(0)
space whose boundary $\pa X_i$ is homeomorphic to a Cantor set, then
$X_i$ is quasi-isometric to a bounded valence bushy tree.
\end{theorem}

Now each of $X_1$, $X_2$ is quasi-isometric to a bushy tree, thus
$X$ is quasi-isometric to the product of two bounded valence bushy trees, and so is
$G$. Therefore we can apply a theorem by Ahlin (\cite{Ahlin} Theorem
1) on quasi-isometric rigidity of lattices in products of trees to
show that a finite index subgroup of $G$ is a lattice in $\Isom(T_1 \times T_2)$ where $T_i$ is a bounded valence bushy tree quasi-isometric to $X_i$, $i=1,2$. Notice that $\Isom(T_1) \times \Isom(T_2)$ is isomorphic to a subgroup of $\Isom(T_1 \times T_2)$ of index 1 or 2 (which can be proved similarly as Lemma \ref{subgroup stabilizing C_1, C_2}), we finally proved the last statement of Theorem \ref{mainthm2}.

\bibliographystyle{alpha}
\bibliography{reference}
\end{document}